\documentclass[english]{article}
\usepackage{preamble}

\title{Chromatic Cardinalities via Redshift}

\begin{document}

\maketitle

\begin{abstract}
    Using higher descent for chromatically localized algebraic $K$-theory, we show that the higher semiadditive cardinality of a $\pi$-finite $p$-space $A$ at the Lubin--Tate spectrum $E_n$ is equal to the higher semiadditive cardinality of the free loop space $LA$ at $E_{n-1}$.
    By induction, it is thus equal to the homotopy cardinality of the $n$-fold free loop space $L^n A$.
    We explain how this allows one to bypass the Ravenel--Wilson computation in the proof of the $\infty$-semiadditivity of the $\Tn$-local categories.
\end{abstract}

\begin{figure}[ht!]
    \centering
    \includegraphics[width=120mm]{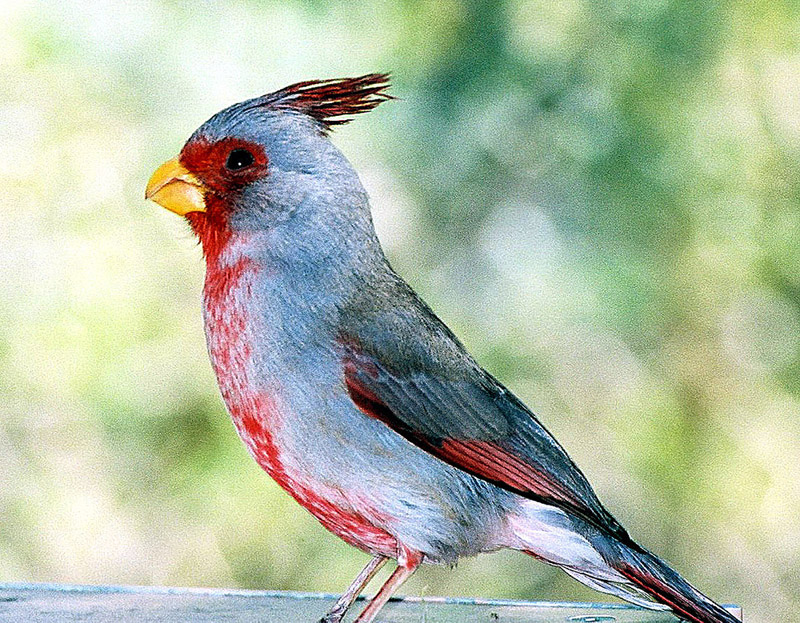}
    \caption*{
        ``\href{https://www.flickr.com/photos/92252798@N07/26805851281/}{Cardinalis sinuatus -- the Pyrrhuloxia}''
        by \href{https://www.flickr.com/photos/92252798@N07/}{Dick Culbert}
        licensed under \href{https://creativecommons.org/licenses/by/2.0/}{CC BY 2.0}.
    }
\end{figure}

\newpage

\section{Introduction}

Natural numbers appear, naturally, as cardinalities of finite sets. Their fundamental role in algebra can be explained from a categorical perspective via the notion of \textit{semiadditivity}. 
A semiadditive category $\cC$ is one in which finite products and finite coproducts canonically coincide (e.g.\ abelian groups). 
When this property holds, every finite set $A$ with cardinality $|A| \in \NN$ induces a natural operation of multiplication-by-$|A|$ on objects $X \in \cC$ via the composition
\[
    X \too[\Delta] \invlim_A X \simeq \colim_A X \too[\nabla] X.
\]
When $\cC$ is monoidal with unit $\one$ (such that the tensor product commutes with finite (co)products in each variable) and $R$ is an algebra object in $\cC$, the map $|A|_R\colon R \to R$ is canonically an $R$-module map, and thus by adjunction is identified with an element of the semiring $\hom(\one,R)$. The resulting semiring homomorphism $\NN \to \hom(\one,R)$ realizes every abstract natural number $n$ as a specific element $n_R$ of $R$. 

In \cite{AmbiKn}, it was observed that the property of semiadditivity sits in a natural hierarchy of \textit{higher semiadditivity} properties, which are most natural to consider in the setting of $\infty$-categories, to which from now on we shall refer simply as \textit{categories}.
In particular, a category $\cC$ is $0$-semiadditive if it is semiadditive in the ordinary sense, and it is $\infty$-semiadditive if, roughly, limits and colimits over $\pi$-finite spaces (i.e.\ those with finitely many connected components, each of which has finitely many non-vanishing homotopy groups all of which are finite)
canonically coincide. As above, this provides a natural multiplication-by-$|A|$ map $|A|_X\colon X \to X$ for every $\pi$-finite space $A$. 
If $\cC$ is moreover monoidal (such that the tensor product commutes with (co)limits over $\pi$-finite spaces in each variable) and $R$ is an algebra object in $\cC$, we obtain an element $|A|_R$ in the semiring $\pi_0\Map(\one, R)$. However, the determination of this element in concrete examples is in general a non-trivial task. 

The most prominent examples of $\infty$-semiadditive categories come from chromatic homotopy theory. Hopkins and Lurie show in \cite{AmbiKn} that for every $n \ge 0$, the Bousfield localization $\SpKn$ of the category of spectra with respect to the height $n$ and (implicit) prime $p$ Morava $K$-theory ring spectrum $K(n)$ is $\infty$-semiadditive. In \cite{TeleAmbi}, it was further shown that the larger telescopic localizations $\SpTn$ are $\infty$-semiadditive as well (and are in fact the maximal ones in a certain precise sense).
In the height $n=0$ case, we have $K(0) = \QQ$ and the corresponding cardinality of every $\pi$-finite space coincides with its Baez--Dolan \textit{homotopy cardinality}:

\begin{example}[{\cite[Proposition 2.3.4]{AmbiHeight}}]\label{htpy-card}
    Working in $\Sp_\QQ$, for every connected $\pi$-finite space $A$ we have
    \[
        |A|_{\Sph_\QQ} = \frac{|\pi_2(A)| |\pi_4(A)| |\pi_6(A)| \cdots}{|\pi_1(A)| |\pi_3(A)| |\pi_5(A)| \cdots}
        \qin \pi_0(\Sph_\QQ) \simeq \QQ.
    \]
\end{example}

In higher chromatic heights the situation is considerably more involved. Note that the prime $p$ is no longer invertible so the above formula does not even make sense. In fact, Yuan and the second author showed in \cite{CarmeliYuan} that already in height $n=1$ the cardinality of a $\pi$-finite space may be \textit{non-rational}:

\begin{example}[{\cite[Theorem A]{CarmeliYuan}}]\label{CY}
    Working in $\Sp_{T(1)} = \Sp_{K(1)}$ at the prime $p = 2$, we have\footnote{In fact, \cite[Theorem A]{CarmeliYuan} determines $|A|$ for all $\pi$-finite $p$-spaces $A$. This was further extended to all $\pi$-finite spaces in Yifan Li's master thesis (University of Copenhagen).}
    \[
        |BC_2|_{\Sph_{K(1)}} = 1 + \varepsilon
        \qin \pi_0(\Sph_{K(1)}) \simeq \ZZ_2[\varepsilon]/(\varepsilon^2, 2\varepsilon).
    \]
\end{example}

In higher heights, not much is known about cardinalities of $\pi$-finite spaces for $\Sph_{K(n)}$, as even the knowledge of the ring $\pi_0(\Sph_{K(n)})$ itself is limited, and for $\Sph_{T(n)}$ the situation is even worse. However, the situation improves drastically if one passes to the algebraic closure of $\Sph_{\Kn}$, namely, the Lubin--Tate ring spectrum $E_n$ associated to the unique formal group over the algebraically closed field $\cl{\FF}_p$. The image of the map 
\[
    \pi_0(\Sph_{\Kn}) \too 
    \pi_0(E_n) \simeq
    W(\cl{\FF}_{p})[[u_1,\dots,u_n]]
\]
is the subring $\ZZ_p = W(\FF_p)$, and since $|A|_{E_n}$ is the image of $|A|_{\Sph_{\Kn}}$ under this map, we can consider it simply as a $p$-adic integer. Furthermore, the kernel of the above map is the nil-radical of $\pi_0(\Sph_{\Kn})$ (see, e.g., \cite[Proposition 2.2.6]{AmbiHeight}), hence the number $|A|_{E_n}$ retains precisely the information in $|A|_{\Sph_{\Kn}}$ modulo nilpotents. Some simple cases can be worked out explicitly:

\begin{example}[{\cite[Lemma 5.3.3]{TeleAmbi}}]\label{Ex_Card_EM}
    Working in $\SpKn$ (or $\SpTn$), for every $d \ge 0$, we have
    \[
        |B^dC_p|_{E_n} = p^{\binom{n-1}{d}}
        \qin \ZZ_p \sseq \pi_0(E_n).
    \]
\end{example}

As mentioned in \cite[Example 2.2.4]{AmbiHeight}, it is possible to show that the cardinalities at the Lubin--Tate spectrum are in fact \textit{natural numbers} and deduce a fairly explicit formula for them.
The idea is to use the technology of tempered ambidexterity developed in \cite{Lurie_Ell3}, extending \cite{hopkins2000generalized, stapleton2013transchromatic}, to express the height $n+1$ cardinality of a space via the height $n$ cardinality of a closely related space.
Applying this inductively, this reduces computations to height $0$.
In more detail, let us denote by $L_p A := \Map(B\ZZ_p, A)$ the $p$-adic free loop space of $A$. 
There is a certain $K(n)$-local $E_{n+1}$-algebra $C_n$ and a natural transchromatic character map
\[
    \chi^A \colon E_{n+1}^A \too C_n^{L_pA}.
\]
The compatibility of these maps with transfers in the variable $A$ implies that $\pi_0(E_{n+1}) \to \pi_0(C_n)$ takes the element $|A|_{E_{n+1}}$ to the element $|L_pA|_{C_n}$ for every $\pi$-finite space $A$. 
One can further show that the cardinalities at $C_n$ and at $E_n$ agree, thus showing that $|A|_{E_{n+1}} = |L_pA|_{E_n}$. Intuitively, the $p$-adic free loop space accounts for the decrease in height.

In the special case where $A$ is further assumed to be a $p$-space (i.e., all homotopy groups are $p$-groups), the $p$-adic free loop space coincides with the ordinary free loop space $LA = \Map(B\ZZ,A)$, in which case $|A|_{E_{n+1}} = |LA|_{E_n}$. In this paper we provide a \textit{different} proof of this special case:

\begin{restatable}{theorem}{cardblue}\label{card-blue}
    Let $A$ be a $\pi$-finite $p$-space. The cardinality $|A|_{E_{n+1}}$ is a natural number and
    \[
        |A|_{E_{n+1}} = |LA|_{E_n} \qin \NN.
    \]
\end{restatable}

Applying this inductively and combining with \cref{htpy-card}, we get:

\begin{cor}\label{card-to-0}
    $|A|_{E_n}$ equals the Baez--Dolan homotopy cardinality of $L^nA$ from \Cref{htpy-card}.
\end{cor}

\begin{rem}
    Is is easy to see that the homotopy cardinality of $LA$ is just the size of the set $\pi_0(A)$ (see \cite[Proposition 2.15]{yanovski2023homotopy}). Thus, for $n \geq 1$ we can reformulate \Cref{card-to-0} as
    \[
        |A|_{E_n} = 
        |\pi_0 \Map((S^1)^{n-1},A)| \qin 
        \NN.
    \]
\end{rem}

In contrast with the transchromatic proof, which proceeds by analyzing the behaviour of the Quillen $p$-divisible group under transchromatic base change, our proof employs the \textit{redshift philosophy} to transport the computation to an analogous, but simpler, categorical setting. Roughly, the transchromatic proof relates cardinalities at $E_{n+1}$ and at $E_n$ by approximating the latter with $C_n$, whereas our proof proceeds instead by approximating the former with $K(E_n)$.

To begin with, the category $\cMod_{E_n}$ of $\Kn$-local $E_n$-modules is an object of the category $\Catpfin$ of categories admitting $\pi$-finite $p$-space indexed colimits and functors that preserve them. 
The latter is $p$-typically $\infty$-semiadditive, allowing us to speak of the cardinality of $A$ at the object $\cMod_{E_n} \in \Catpfin$. This turns out to be simply the colimit of the constant $A$-shaped diagram on $E_n$, that is
\[
    |A|_{\cMod_{E_n}} \simeq A \otimes E_n
    \qin \cMod_{E_n}.
\]
Furthermore, the object $A\otimes E_n$ is dualizable, and its symmetric monoidal dimension is given by
\[
    \dim(A \otimes E_n) = |LA|_{E_n}
    \qin \pi_0(E_n).
\]

Note that while $\Catpfin$ is not stable, so we can not measure the chromatic height of the object $\cMod_{E_n}$, we can measure its \textit{semiadditive height}, which by the semiadditive redshift theorem is $n+1$ (see \cite[Theorem B]{AmbiHeight}).
Thus, the combination of the two displayed formulas above bears a close resemblance to \Cref{card-blue}. 
The bridge between the categorical story and the chromatic one is the algebraic $K$-theory functor, which produces a spectrum from a (stable) category. The key ingredient in our proof is the higher descent property of chromatically localized algebraic $K$-theory \cite[Theorem A]{cycloredshift}. It implies that the map 
\[
    \pi_0(\cMod_{E_n}^{\dbl}) \too 
    \pi_0(L_{T(n+1)}K(E_n))
\]
sending a dualizable module $M$ to its class $[M]$ in ($T(n+1)$-localized) $K$-theory preserves cardinalities of $\pi$-finite $p$-spaces. To conclude, we observe that by chromatic redshift, $L_{T(n+1)}K(E_n)$ is a non-zero commutative algebra and hence by the chromatic nullstellensatz admits a map of $T(n+1)$-local commutative algebras to (a mild extension of) $E_{n+1}$. Putting everything together we get
\[
    |A|_{E_{n+1}} = 
    |A|_{L_{T(n+1)}K(E_n)} = 
    \dim(A \otimes E_n) =
    |LA|_{E_n}.
\]

\begin{rem}
    To complete the picture and tighten the analogy, the role of the transchromatic character map is played by the ordinary character map
    \[
        \chi^A \colon 
        (\cMod_{E_n}^{\dbl})^A \too 
        E_n^{LA}
    \]
    taking a local system $V$ of dualizable $\Kn$-local $E_n$-modules on $A$ to its character $\chi^A_V \colon LA \to E_n$ whose value on a loop $\gamma \in LA$ is given by $\tr(\gamma \mid V)$. 
    Though not used in the present paper, it is true that the computation $\dim(A \otimes E_n) = |LA|_{E_n}$ is a special case of the compatibility of $\chi^A$ with higher semiadditive transfers. This is part of an upcoming work of Cnossen, Ramzi and the second and fourth authors.
\end{rem}

The symmetric monoidal dimension $\dim(A\otimes E_n)$ identifies with the perhaps more familiar \textit{Morava--Euler characteristic} (see \Cref{dim_dim})
\[
    \chi_{n}(A) :=
    \dim_{\FF_p}(K(n)^{0}(A)) - 
    \dim_{\FF_p}(K(n)^{1}(A)).
\]
Thus, using the relation $\dim(A \otimes E_n) = |LA|_{E_n}$ between cardinality and symmetric monoidal dimension, \Cref{card-blue} implies the following corollary which is stated in more concrete terms:

\begin{cor}[{cf.\ \cite[Corollary 4.8.6]{Lurie_Ell3}}]
    Let $A$ be a $\pi$-finite $p$-space. We have,
    \[
        \chi_n(A) = 
        |\pi_0 \Map((S^1)^n,A)|
        \qin \NN.
    \]
\end{cor}

\begin{rem}
    In \cite{yanovski2023homotopy}, the fourth author used the above corollary to show that the sequence $|A|_{E_n}$, as a function of $n$, is $\ell$-adically continuous for every prime $\ell \mid p-1$. Or equivalently, that the sequence of natural numbers $\chi_n(A)$ is $\ell$-adically continuous and extrapolates to the homotopy cardinality of $A$ at $n = -1$.
\end{rem}

Given a $\pi$-finite $p$-space $A$, computing the homotopy cardinality of the iterated free loop space $L^nA$ might be quite difficult. There is however a family of examples, generalizing  \Cref{Ex_Card_EM}, in which it can be expressed in elementary terms. If $A$ happens to be a loop space, we have a decomposition $LA \simeq A \times \Omega A$. Since the homotopy groups of $\Omega A$ are the same as those of $A$ shifted by 1, one can easily deduce by induction the following:

\begin{example}[{cf.\ \cite[Example 3.1]{yanovski2023homotopy}}]\label{Ex_Card_Loop}
    Working in $\SpKn$ (or $\SpTn$), for every $\pi$-finite $p$-space $A$ which is a loop space we have
    \[
        |A|_{E_n} = 
        \prod_{k \ge 0} |\pi_k(A)|^{\binom{n-1}{k}}.
    \]
\end{example} 

We conclude with an observation regarding the proof of the $\infty$-semiadditivity of the monochromatic stable categories.
The original proof of the $\infty$-semiadditivity of $\SpKn$ due to Hopkins--Lurie \cite{AmbiKn} relies heavily on a careful analysis of the seminal computation of Ravenel--Wilson of $\Kn_*(B^dC_p)$ as a Hopf algebra \cite{ravenelwilson1980}, and an integral lift of this computation to $E_n$. In contrast, the proof of the $\infty$-semiadditivity of $\SpTn$ (hence in particular of $\SpKn$) in \cite{TeleAmbi} avoids such an elaborate analysis, and relies only on the computation of the Morava--Euler characteristics of the spaces $B^dC_p$ (in fact, only that they are rational and non-zero, see \cite[Corollary 5.1.9 and Theorem 5.3.1]{TeleAmbi}). While at the time it was not known whether one can prove this weaker property of the spaces $B^dC_p$ without invoking the full power of the Ravenel--Wilson computation, we can now close the circle by providing an alternative route.

\begin{obs}
    The proof of the $m$-semiadditivity of $\Sp_{T(n)}$ in \cite{TeleAmbi} requires only the knowledge of the Morava--Euler characteristic of $B^dC_p$ for $1\le d \le m-1$.
    Furthermore, the proof of \Cref{card-blue} for $d$-finite $p$-spaces, and the subsequent computation of the Morava--Euler characteristic of $B^dC_p$, depends only on the $d$-semiadditivity of $\Sp_{T(n)}$.
    Thus, by preforming all of the arguments under an inductive hypothesis on $m$, we can compute the Morava--Euler characteristic of $B^dC_p$ in the required range as above, hence, in particular, eliminating the reliance on the Ravenel--Wilson computation.
\end{obs}

\section{The Proof}

We now carry out in the detail the proof of \cref{card-blue} sketched in the introduction.
Given a small (symmetric monoidal) stable idempotent complete category $\cC$, we can consider its algebraic $K$-theory (commutative ring) spectrum $K(\cC)$. 
For a (commutative) ring spectrum $R$, one usually defines $K(R)$ as $K(\cC)$ for $\cC$ the (symmetric monoidal) category $\Mod_R^\omega$ of perfect $R$-module spectra. 
However, when $R$ is $T(n)$-local, we can also consider the algebraic $K$-theory of the category $\cMod_R^\dbl$ of dualizable $T(n)$-local $R$-modules. 
By \cite[Proposition 4.15]{clausen2020descent}, these two algebraic $K$-theory spectra identify after $T(n+1)$-localization. 
We shall henceforth adopt the definition in terms of dualizable $T(n)$-local modules and use the notations
\begin{gather*} 
    \KTnp(\cC) := \LTnp(K(\cC)), \\
    \KTnp(R) := \LTnp(K(R)) \simeq \LTnp(K(\cMod_R^\dbl)).
\end{gather*}
By construction, every $M\in \cMod_R^\dbl$ defines a class 
\[
    [M] \qin \pi_0(\KTnp(R)).
\]

We briefly recall the setup of the higher descent result for $\KTnp$ from \cite{cycloredshift}. Let $\CatLnf$ be the category of small $\Lnf$-local stable idempotent complete categories and exact functors. In \cite[Theorem A]{cycloredshift} we show that the functor
\[
    \KTnp \colon \CatLnf \too \SpTnp
\]
preserves limits and colimits of $\pi$-finite $p$-space shape. 
We let $\Catpfin \subset \Cat$ be the subcategory on those categories which admit $\pi$-finite $p$-space colimits and functors preserving them, and set
\[
    \CatLnfpfin := \CatLnf \cap \Catpfin.
\]

\begin{prop}\label{KTnp_Sadd}
    The category $\CatLnfpfin$ is $p$-typically $\infty$-semiadditive and the restricted functor
    \[
        \KTnp \colon \CatLnfpfin \too \SpTnp
    \]
    is $p$-typically $\infty$-semiadditive (i.e., preserves $\pi$-finite $p$-space limits and colimits).
\end{prop}

\begin{proof}
    The category $\Catpfin$ is $p$-typically $\infty$-semiadditive by \cite[Proposition 2.2.7]{AmbiHeight} (which immediately follows from \cite[Proposition 5.26]{harpaz2020ambidexterity}).
    The inclusion $\CatLnfpfin \hookrightarrow \Catpfin$ preserves limits, thus $\CatLnfpfin$ is also $p$-typically $\infty$-semiadditive by \cite[Proposition 2.1.4(3)]{AmbiHeight}. Similarly, the second inclusion $\CatLnfpfin \hookrightarrow \CatLnf$ preserves limits.
    Thus, by the higher descent theorem \cite[Theorem A]{cycloredshift}, the composition
    \[
        \CatLnfpfin \too
        \CatLnf \oto{\KTnp}
        \SpTnp
    \]
    preserves $\pi$-finite $p$-space limits hence is $p$-typically $\infty$-semiadditive.
\end{proof}

We deduce the following fundamental identity:

\begin{prop}\label{card-K}
    Let $R \in \CAlg(\SpTn)$ and let $A$ be a $\pi$-finite $p$-space, then
    \[
        |A|_{\KTnp(R)} = [A\otimes R]
        \qin \pi_0(\KTnp(R)).
    \]
\end{prop}

\begin{proof}
    A $p$-typically $\infty$-semiadditive functor preserves cardinalities of $\pi$-finite $p$-spaces by \cite[Corollary 3.2.7]{TeleAmbi}.
    By \cite[Proposition 2.54 and Proposition 4.15]{moshe2021higher} the category $\cMod_R^\dbl$ is stable and has $\pi$-finite $p$-space indexed colimits (in fact, it is $p$-typically $\infty$-semiadditive), and is clearly $\Lnf$-local, i.e., it is in $\CatLnfpfin$. 
    Since the inclusion $\CatLnfpfin \into \Catpfin$ 
    is $p$-typically $\infty$-semiadditive, cardinalities in the source are computed as in the target, so in particular, by \cite[Proposition 7.6]{moshe2021higher}, we have
    \[
        |A|_{\cMod_R^\dbl} = A\otimes R \qin 
        \pi_0(\cMod_R^\dbl).
    \]
    Combining this with \Cref{KTnp_Sadd}, which implies that $\KTnp$ preserves cardinalities of $\pi$-finite $p$-spaces, we get
    \[
        |A|_{\KTnp(R)} = [A\otimes R] \qin 
        \pi_0(\KTnp(R)).
    \]
\end{proof}

We now specialize to the case of the Lubin--Tate spectrum $R = E_n$. Since $E_n$ is even-periodic with $\pi_0(E_n)$ a complete regular local ring, we have by \cite[Proposition 10.11]{AkhilGalois} that $\cMod_R^\dbl = \Mod_R^\omega$. Thus, the above considerations regarding algebraic $K$-theory of perfect vs.\ dualizable modules apply even before $T(n+1)$-localization. Moreover, in this case, the equality of \Cref{card-K} is essentially of integers. 

\begin{prop}\label{K0_Z}
    For every $n$, we have an isomorphism of rings
    \(
       \pi_0(K(E_n)) \simeq \ZZ.
    \)
\end{prop}

\begin{proof}
    We observe that in \cite[Proposition 10.11]{AkhilGalois}, the proof of the ``local and dualizable implies perfect'' direction does not use retracts. Thus, a perfect $E_n$-module is just an iterated cofiber of finite free $E_n$-modules, hence its class in $\pi_0(K(E_n))$ is an integer multiple of the unit $[E_n]$.
    Furthermore, the classes $k[E_n]$ are distinct from one another since they are distinguished by the dimension map $\pi_0(K(E_n)) \to \pi_0(E_n)$.
\end{proof}

\begin{rem}\label{dim_dim}
    \Cref{K0_Z} implies the well known fact that
    \[
        \dim_{E_n}(M) = 
        \dim_{\FF_p}(\pi_0(K(n)\otimes_{E_n}M)) - 
        \dim_{\FF_p}(\pi_1(K(n)\otimes_{E_n}M))
    \]
    for every dualizable $K(n)$-local $E_n$-module $M$. Indeed, both sides of the equation are additive invariants, hence factor through $\pi_0(K(E_n)) \simeq \ZZ$, and agree on the unit $[E_n]$. 
\end{rem}

We are now in position to prove our main theorem.

\begin{proof}[Proof of \cref{card-blue}]
    By \cref{K0_Z}, a class $[M] \in \pi_0(K(E_n))$ is just an integer, which can be further identified with $\dim(M) \in \pi_0(E_n)$. Thus, combining \cref{card-K} and \cite[Corollary 3.3.10]{TeleAmbi}, we have an equality of integers
    \[
        |A|_{\KTnp(E_n)} = [A \otimes E_n] = \dim(A \otimes E_n) = |LA|_{E_n}.
    \]
    It remains to observe that since $\KTnp(E_n) \neq 0$ by the seminal redshift result of \cite[Theorem A]{yuan2021examples}, we can apply the chromatic nullstellensatz in the form of \cite[Theorem D]{Null}, to get a map of commutative algebras 
    \[
        \KTnp(E_n) \too E_{n+1}(\kappa)
    \]
    where $E_{n+1}(\kappa)$ is the Lubin--Tate spectrum associated to the unique formal group over some algebraically closed field $\kappa$. Finally, both this map and the map $E_{n+1} \to E_{n+1}(\kappa)$ induced by $\cl{\FF}_p \into \kappa$ preserve cardinalities, so we get 
    \[
        |A|_{E_{n+1}} = |A|_{\KTnp(E_n)} = |LA|_{E_n}.
    \]
\end{proof}

\section{Acknowledgements}

We thank the anonymous referees for their helpful comments.
The second author is partially supported by the Danish National Research Foundation through the Copenhagen Centre for Geometry and Topology (DNRF151).
The third author was supported by ISF1588/18, BSF 2018389 and the ERC under the European Union's Horizon 2020 research and innovation program (grant agreement No. 101125896). The fourth author was supported by ISF1848/23.

\bibliographystyle{alpha}
\phantomsection\addcontentsline{toc}{section}{\refname}
\bibliography{cyclored}

\end{document}